\documentclass[12pt]{amsart}

\usepackage{scrtime}

\usepackage{mathrsfs}
\usepackage{amsfonts}
\usepackage{amssymb}
\usepackage{amsfonts, amscd, amsmath, mathrsfs, amssymb, amsthm, amsxtra, bbding, epsfig, graphicx, latexsym, url}
\usepackage[papersize={8.3in,11.8in},textwidth=16cm,textheight=24cm,centering]{geometry}
\usepackage{enumerate}
\usepackage{caption} 

\usepackage{graphicx}
\usepackage{tikz}
\usepackage{float} 

\usepackage{todonotes}

\usepackage{xcolor}
\definecolor{cite}{rgb}{0.00,0.60,1.00}
\definecolor{url}{rgb}{1.00,0.10,0.80}
\definecolor{link}{rgb}{0.00,0.00,1.00}
\usepackage[colorlinks,linkcolor=link,urlcolor=url,citecolor=cite,pagebackref,breaklinks]{hyperref}

\hypersetup{
pdfstartpage=1,
pdfstartview=FitH}

\def\\{\cr}
\def\({\left(}
\def\){\right)}

\DeclareFontFamily{U}{mathx}{\hyphenchar\font45}
\DeclareFontShape{U}{mathx}{m}{n}{
      <5> <6> <7> <8> <9> <10>
      <10.95> <12> <14.4> <17.28> <20.74> <24.88>
      mathx10
      }{}
\DeclareSymbolFont{mathx}{U}{mathx}{m}{n}
\DeclareMathAccent{\widecheck}{\mathalpha}{mathx}{"71}

\numberwithin{equation}{section}

\allowdisplaybreaks

\newtheorem{theorem}{Theorem}[section]
\newtheorem{lemma}[theorem]{Lemma}

\newtheorem{corollary}[theorem]{Corollary} 

\newtheorem{remark}[theorem]{Remark}

\makeatletter
\newcounter{roem}
\renewcommand{\theroem}{\Roman{roem}}

\newcommand{\c@org@eq}{}
\let\c@org@eq\c@equation
\newcommand{\org@theeq}{}
\let\org@theeq\theequation

\newcommand{\setroem}{
\let\c@equation\c@roem
 \let\theequation\theroem}

\newcommand{\setarab}{
\let\c@equation\c@org@eq
\let\theequation\org@theeq}
\makeatother

\newtheorem*{claim*}{Claim}

\newcommand{\GL}{\mathrm{GL}}

\newcommand{\Disc}{\mathrm{Disc}}

\usepackage{scalerel}

\renewcommand{\geq}{\geqslant}
\renewcommand{\leq}{\leqslant}
\renewcommand{\ge}{\geqslant}
\renewcommand{\le}{\leqslant}

\DeclareMathOperator{\Mod}{mod}

\renewcommand{\bmod}[1]{\,(\Mod{ #1})}

\newcommand{\F}{\mathbb{F}}
\newcommand{\N}{\mathbb{N}}

\newcommand{\R}{\mathbb{R}}
\newcommand{\Z}{\mathbb{Z}}

\newcommand{\balpha}{\boldsymbol{\alpha}}
\newcommand{\bbeta}{\boldsymbol{\beta}}

\newcommand{\btheta}{\boldsymbol{\vartheta}}

\newcommand{\cA}{\mathcal{A}}
\newcommand{\cB}{\mathcal{B}}
\newcommand{\cC}{\mathcal{C}}
\newcommand{\cD}{\mathcal{D}}
\newcommand{\cE}{\mathcal{E}}

\newcommand{\cI}{\mathcal{I}}

\newcommand{\cQ}{\mathcal{Q}}

\newcommand{\ff}{\mathfrak{f}}

\newcommand{\fC}{\mathfrak{C}}

\newcommand{\fR}{\mathfrak{R}}
\newcommand{\fS}{\mathfrak{S}}

\newcommand{\fY}{\mathfrak{Y}}

\newcommand\e{{\mathbf{\,e}}}

\def\ssum{\mathop{\sum\ldots \sum}}

\def\mand{\qquad\mbox{and}\qquad}

\def\le{\leqslant}
\def\leq{\leqslant}
\def\ge{\geqslant}
\def\geq{\geqslant}

\usepackage{graphicx}
\usepackage{tikz}

\begin{document}

\title[Tridimensional character sums and applications] 
{Tridimensional character sums with polynomial arguments and applications}

\author[\'E. Fouvry] {\'Etienne Fouvry}
\address{D{\'e}partement de Math{\'e}matiques, Universit{\'e} Paris-Saclay,
91405 Orsay Cedex, 
France}
\email{etienne.fouvry@universite-paris-saclay.fr}

\author[I. E. Shparlinski] {Igor E. Shparlinski}
\address{School of Mathematics and Statistics, University of New South Wales, Sydney NSW 2052, Australia}
\email{igor.shparlinski@unsw.edu.au}

\author[P. Xi]{Ping Xi}

\address{School of Mathematics and Statistics, Xi'an Jiaotong University, Xi'an 710049, P. R. China}
\email{ping.xi@xjtu.edu.cn}

\begin{abstract} 
Let $p$ be a large prime and $\chi$ a non-trivial Dirichlet character modulo $p$. We study the character sum
\[
\sum_{a \sim A} \sum_{b \sim B} \sum_{c \sim C} \alpha(a,b) \beta(c)\chi(f(a) + bc),
\]
where  $z \sim Z$ means $Z\le z < 2Z$,  $f \in \Z[X]$ is of small degree $k$, and $\balpha=(\alpha(a,b))_{a\sim A,b\sim B}$ 
and $\bbeta=(\beta(c))_{c\sim C}$ are two complex coefficients. We prove non-trivial upper 
bounds for this sum in either of the two cases:
\begin{itemize}
\item $\bbeta\equiv1$, $k=2,3,4,5$ and $A,B,C>p^{\frac{1}{8}+\varepsilon}$, 
\item $\bbeta$ general, $k=2,3$ and $A,B,C>p^{\frac{1}{6}+\varepsilon}$,
\end{itemize}
where $\varepsilon>0$ is fixed.

This work was originally motivated by an intermediate result of Ganguly and Rajan (2023) on counting $2\times2$ matrices over $\F_p$ with irreducible characteristic polynomials, where the entries are in short segments. The new bounds here allow us to count such matrices in much shorter segments.
\end{abstract}

\keywords{Character sum over polynomials}
\subjclass[2020]{11L10, 11L40}

\date{\today,\ \thistime}

\maketitle

\tableofcontents

\section{Introduction and statement of the results} 
 
 \subsection{Set-up}
 
Let $p$ be a large prime and $\chi$ a non-trivial Dirichlet character modulo $p$. Give a polynomial $f \in \Z[X]$ and  real $A,B,C > 1$. For two complex coefficients $\balpha=(\alpha(a,b))_{a\sim A,b\sim B}$ and $\bbeta=(\beta(c))_{c\sim C}$,
we consider the character  sum
\begin{equation}\label{eq:charactersum}
\cC_{f,\chi}(\balpha,\bbeta) =\sum_{a \sim A} \sum_{b \sim B} \sum_{c \sim C} \alpha(a,b) \beta(c)\chi(f(a) + bc),
\end{equation}
where $a \sim A$ means $A \le a < 2A$ and similarly for $b \sim B$ and $c \sim C$. 
The variables $a,b$ are attached with a general coefficient $\balpha$, although they are separated in the character. This explains we say unbalanced in the title.

The study of $\cC_{f,\chi}(\balpha,\bbeta)$ is a continuation of our work~\cite{FSX25}, in which we deal with the case of linear polynomial $f$.
This may remind you some pioneering works by Karatsuba~\cite{Ka92}, Friedlander and Iwaniec~\cite{FI93} on the double character sum
\[
\sum_{a\in\cA}\sum_{b\in \cB}\chi(a+b),
\]
where $\chi$ is a non-trivial Dirichlet character modulo $p$, and $\cA,\cB$ are structured sets (e.g., intervals) or quite general subsets in $\Z$. We refer the readers to our preceding paper~\cite{FSX25} for more precise references.

If one of $A,B,C$ in~\eqref{eq:charactersum} is larger than $p^{1/2+\varepsilon}$, one can employ the Cauchy--Schwarz inequality to smooth out one variable, and Weil's bound for complete character sums applies after a completing argument in the P\'olya--Vinogradov method, producing  a non-trivial upper bound for $\cC_{f,\chi}(\balpha,\bbeta)$.
On the other hand, we can also appeal to Burgess's bound as soon as $\beta(c)\equiv1$ for $c\sim C$, in which case we can bound $\cC_{f,\chi}(\balpha,\bbeta)$ non-trivially for $C>p^{1/4+\varepsilon}$.
The aim of this paper is, by introducing some variants of the original Burgess method, to give non-trivial bounds for $\cC_{f,\chi}(\balpha,\bbeta)$ under mild conditions on the weights $\balpha,\bbeta$, and with $A,B,C$ as small as possible compared to $p$.
Note that the case of $f(X) = X$ corresponds to the character sums estimated 
in~\cite[Theorems~2.1 and~2.5]{FSX25}.  Hence we always assume that $\deg f \ge 2$.
 
Since we are mostly interested in short sums, for convenience, throughout this work  we always  make the following assumption:
{\it The  positive parameters $A$, $B$ and $C$ satisfy 
\begin{equation}\label{eq:Cond ABC}
A^k \le \gamma_f p, \qquad A  \ge 1/\gamma_f,   \qquad  BC\le p/73, 
\end{equation}
for some sufficiently small constant $\gamma_f$ depending only on $f\in \Z[X]$ and $k=\deg f\ge 2$.\/}

 We frequently use a trivial observation that~\eqref{eq:Cond ABC}, assuming  that $\gamma_f$ is small enough, 
implies $p\nmid f(a)$, and also $p\nmid (f(a)+bc)$ when we have $f(a)+bc\neq0$, for all $a\sim A,b\sim B$ and $c\sim C$.  

We are ready to formulate our first result
for a constant weight $\bbeta = \mathbf 1$,  that is, for $\beta(c) =1$, $c \sim C$.

\begin{theorem}\label{thm:TypeI}
Let $f\in \Z[X]$ be a fixed  polynomial of degree $k \ge 2$ and 
let  $A,B,C\geqslant2$ satisfy~\eqref{eq:Cond ABC}  with a sufficiently small $\gamma_f > 0$. 
Assume that $r$ is an arbitrary fixed positive integer and $C > 4p^{\frac{1}{2r}}$. 
For  $\|\balpha\|_{\infty} \leqslant1 $ and $\bbeta = \mathbf 1$  we have
\[
|\cC_{f,\chi} (\balpha,\bbeta)|\le ABCp^{o(1)}\cdot \Delta_1,
\]
where
\[
\Delta_1= (AB)^{-\frac{1}{2r}}C^{-\frac{1}{r}}p^{\frac{r+1}{4r^2}}\(\frac{A}{B}+1\)^{\frac{1}{2r}} \( \frac{(A^k+BC)BC}{p^{1+\frac{1}{2r}}} + 1\)^{\frac{1}{2r}}.
\]
\end{theorem}

\begin{corollary}\label{cor:TypeI}
Let $\|\balpha\|_{\infty} \leqslant1 $, $\bbeta = \mathbf 1$ and  $f\in \Z[X]$ a fixed   polynomial of degree $k \ge 2$. 
For all  $A, B, C = p^{\eta+o(1)}$, with any fixed 
\begin{equation}\label{eq:Cond N1}
 \eta \in \left] \frac{1}{8}, \frac{1}{2(k-2)}\right[
\end{equation}
we have
\[
|\cC_{f,\chi} (\balpha,\bbeta) |
 \le ABCp^{-\delta}
\]
for some $\delta>0$, depending only on $\eta$, provided that $p$ is large enough.
\end{corollary}

\begin{remark}
Clearly the condition~\eqref{eq:Cond N1} limits the degree $k$ of $f$ to $2\leqslant k\leqslant5$. On the other hand, 
the upper bound in~\eqref{eq:Cond N1} becomes $\infty$ for $k=2$, which is relevant for our applications in Theorem~\ref{thm:cE(N,p)-asymptoticformula} below.
 \end{remark}

The next result deals with general coefficients $\balpha$ and $\bbeta$.

\begin{theorem}\label{thm:TypeII}
Let $f\in \Z[X]$ a fixed   polynomial of degree $k \ge 2$ and 
let  $A,B,C\geqslant2$ satisfy~\eqref{eq:Cond ABC}  with a sufficiently small $\gamma_f > 0$. 
Assume that $r$ is an arbitrary fixed positive integer and $C > 4p^{\frac{1}{r}}$. 
For   $\|\balpha\|_{\infty}, \|\bbeta\|_{\infty}  \leqslant1 $, we have
\[
|\cC_{f,\chi} (\balpha,\bbeta) | \le ABCp^{o(1)}\cdot \Delta_2,
\]
where
\[
\Delta_2=(ABC)^{-\frac{1}{2r}}p^{\frac{r+1}{4r^2}+o(1)}\(\frac{A}{B}+1\)^{\frac{1}{4r}} \( \frac{(A^k+BC)BC}{p^{1+\frac{1}{r}}} + 1\)^{\frac{1}{2r}}.
\]
\end{theorem}

In particular, for bounded weights and ranges of the same size,  we have the following:

\begin{corollary}\label{cor:TypeII}
Let $\|\balpha\|_{\infty},\|\bbeta\|_{\infty} \leqslant1$ and  $f\in \Z[X]$ a fixed polynomial of degree $k =2,3$. 
For all $A, B, C = p^{\eta+o(1)}$, with any fixed 
\begin{equation}\label{eq:Cond N2}
 \eta \in \left] \frac{1}{6}, \frac{1}{2(k-1)}\right[
\end{equation}
we have
\[
|\cC_{f,\chi} (\balpha,\bbeta) |
 \le ABCp^{-\delta}
\]
for some $\delta>0$, depending only on $\eta$, provided that $p$ is large enough.
\end{corollary}

\begin{remark}
Theorems~\ref{thm:TypeI} and~\ref{thm:TypeII} deal with  polynomials   $f \in \Z[X]$. The arguments in this paper also allow us to extend these results to rational functions of small degrees, but with small savings. 
\end{remark}

\subsection{Counting elliptic matrices over $\F_p$}
We  identify $\F_p$ with the set of integers $\{0,1,  \ldots,  p-1\}$. 
Denote by $M_2(\F_p)$ the set of all $2\times2$ matrices with entries in $\F_p$.
Following Ganguly and Rajan~\cite[Section~1.4.2]{GR23}, a matrix $A\in M_2(\F_p)$ is called {\it elliptic\/} if its characteristic polynomial $f_A$ is irreducible. Let $\cE(p)$
the set of all elliptic matrices in $M_2(\F_p)$.

Clearly the irreducibility of the characteristic polynomial $f_A$ of a matrix 
\[
A  = \begin{pmatrix}  a & b \\ c   & d 
 \end{pmatrix} \in M_2(\F_p)
\]
is determined by the 
quadratic residuosity of its discriminant $\Disc(f_A)$ (for $p\ge 3$). More precisely, the above $A$ is elliptic if and only if 
\[
\Disc(f_A) = (a-d)^2 + 4bc
\]
is a quadratic non-residue modulo $p$. This easily implies that the set $\cE(\F_p)$ is of cardinality
\[
\sharp \cE(p)=\frac{1}{2}p^4 + O(p^3).
\]

We now examine the elliptic matrices in $\cE(p)$ such that the entries fall into short segments. To this end, define
\[
\cE(N,p) 
=\left\{ 
 \begin{pmatrix}  a  & b  \\ c   & d 
 \end{pmatrix}\in\cE(p):~ 1\leqslant a,b,c,d\leqslant N 
\right\}
\]
for $1\leqslant N\leqslant p$. It is reasonable to expect that
\begin{equation}\label{eq:cE(N,p)-expected}
\sharp\cE(N,p)\sim \frac{1}{2}N^4
\end{equation}
as long as $N$ is not too small compared to $p$.
As a step in their paper, Ganguly and Rajan~\cite{GR23} initiated the study of $\sharp \cE(N,p)$, and  proved~\eqref{eq:cE(N,p)-expected}
for all $N>p^{1/2+\varepsilon}$ with any $\varepsilon > 0$, by appealing to the P\'olya--Vinogradov bound~\cite[Theorems~12.5]{IK04} for incomplete character sums. They also noticed that
the Burgess bound~\cite[Theorem~12.6]{IK04} allows to reduce this threshold to $N \ge p^{1/4+\varepsilon}$ (see~\cite[Remark~1.14]{GR23}). 

Corollary~\ref{cor:TypeI} allows us to establish the asymptotic formula~\eqref{eq:cE(N,p)-expected} with a power saving as long as $N>p^{1/8+\varepsilon}$.

\begin{theorem}\label{thm:cE(N,p)-asymptoticformula}
For all $N  \ge  p^{\frac{1}{8}+\varepsilon}$ with any fixed $\varepsilon>0$, there exists some constant $\delta>0$, depending only on 
$\varepsilon$,  such that
\[
\sharp\cE(N,p)= \frac{1}{2}N^4+O(N^{4-\delta}).
\]
\end{theorem}

Theorem~\ref{thm:cE(N,p)-asymptoticformula} is an immediate consequence of Theorem~\ref{thm:TypeI} with $f(X)=X^2$, and Theorem~\ref{thm:TypeII} also allows us to count elliptic matrices with weights on the entries.

\subsection{Further applications} Clearly, our bounds in Theorems~\ref{thm:TypeI} and~\ref{thm:TypeII}, combined with the 
Vaughan identity~\cite{Vau}, can be used to estimate character sums of the form 
\[
\sum_{a \sim A} \sum_{\substack{q \sim Q\\q~\text{prime}}} \vartheta(a) \chi(f(a) + q) \mand
\sum_{a \sim A} \sum_{n \sim N} \vartheta(a) \mu(n) \chi(f(a) + n),
\]
with arbitrary weight $\btheta = (\vartheta(a))_{a\sim A}$ of bounded $L_\infty$-norm, where $\mu(n)$ 
is the M{\"o}bius function.  In particular, if weight $\vartheta(a)$ is supported only on primes, one can get results about the distribution 
of quadratic residues of the form $q_1^2 + q_2$ for primes $q_1 \le Q_1$, $q_2 \le Q_2$ and rather small $Q_1$ and $Q_2$.

\section{Preliminaries}

\subsection{Notation} Throughout the whole article, as usual, the notations 
$U = O(V)$, $U \ll V$, and $ V\gg U$ are equivalent to $|U|\le c V$ for some 
positive constant $c$, which, throughout this work, is allowed to depend on 
the polynomial $f \in \Z[X]$ and some other parameters which we declare as fixed.

We use  $\tau(n)$ to denote the number of distinct  positive integer divisors of an integer $n\ne 0$, for which we 
recall the following 
well-known estimate  
\begin{equation}
\label{eq:tau}
\tau(k) = k^{o(1)}, \quad k \to \infty, 
\end{equation}
see, for example,~\cite[Equation~(1.81)]{IK04}).

We also denote $\e(z) = \exp(2 \pi i z)$. 

As we have already mentioned,  for $Z \ge 1$, we  write $z \sim Z$ to denote that $Z \le  z <  2Z$.

For an integer $k$ with $\gcd(k,p)=1$ we denote by $\overline{k}$ the multiplicative inverse of $k$ modulo $p$, 
that is, $k \overline{k} \equiv 1  \bmod p$. 

Let $\Xi_p$ be the set of all multiplicative characters modulo $p$ and
let $\Xi_p^* = \Xi_p \setminus \{\chi_0\}$ be the set of nonprincipal
characters.

   \subsection{Moments of short character sums}

We also make use 
of  the following  version of the  {\it Davenport--Erd\H{o}s Lemma} to refer  
to the seminal result in~\cite[Lemma~3 and Footnote p.~262]{DE52}, 
which we present it in the form given by~\cite[Lemma~2.3]{FS23}.  
The proof depends heavily on Weil's bound for complete character sums, as a consequence of his resolution of the Riemann Hypothesis for curves over finite fields.
 
\begin{lemma}\label{lm:DavenportErdos}
Let $\chi \in \Xi_p^*$. 
For any set $\cD \subseteq \F_p$ with  $\sharp\cD = D$, for any complex 
weights $\balpha$ satisfying $\vert \alpha_d \vert \leq 1$ for $d \in \cD$ and for every integer  $r \geq 1$,  we have the inequality
\[
  \sum_{x\in\F_p}  
   \left|\sum_{d \in \cD}\alpha_d \chi(x+d)\right|^{2r} \le 
    \(2r D\)^r p + 2r D^{2r} p^{1/2}.
  \]
\end{lemma}

We remark that the condition  $\cD \subseteq \F_p^*$ in~\cite[Lemma~2.3]{FS23}
can be replaced with  $\cD \subseteq \F_p$  without any changes in the argument 
or result.

We also need the following version of Lemma~\ref{lm:DavenportErdos}, which is also a consequence of Weil's bound for complete multiplicative 
character sums over finite fields. This was originally formulated in~\cite[Lemma~3.2]{FSX25}.

\begin{lemma}\label{lm:DavenportErdos2}
Let $\chi_1,\chi_2 \in \Xi_p^*$. 
For any set $\cD \subseteq \F_p$ with  $\sharp\cD = D$, for any complex 
weights $\balpha$ satisfying $\vert \alpha_d \vert \leq 1$ for $d \in \cD$ and for every integer  $r \geq 1$,  we have the inequality
$$
\mathop{\sum\sum}_{x,y\in\F_p}\left| \sum_{d\in\cD}\alpha_d\chi_1(x+d)\chi_2(y+d)\right| ^{2r}\leqslant (2r)^r(D^rp^2+2rD^{2r}p).
$$
\end{lemma}
 
  \subsection{Bounding the number of solutions to polynomial congruences}

An integer number $q$ is called
\begin{itemize}
\item  {\it $r$-th power-free\/} if   any prime number $p \mid q$ satisfies $p^r \nmid q$; 

\item  {\it  $r$-th power-full\/} if any prime number  $p \mid q$ satisfies $p^r \mid q$. 
\end{itemize}
We note that $1$ is both $r$-th power-free and $r$-th power-full for any $r \in \N$, and 
for $r=2$ such numbers are called {\it square-free\/} and {\it square-full\/}, respectively.

For a polynomial $F\in \Z[X]$  we denote by 
$I_F(q)$ the number of solutions to the congruence 
$F(x) \equiv 0 \bmod q$ with $x\bmod q$.  

Furthermore, we define
\begin{align*}
J_k(q) = \max\{I_F(q):~ F(X) = b_kX^k&+\cdots+b_1X+b_0, \\
&  \gcd(q, b_1, \cdots, b_k) = 1\}. 
\end{align*}
We note that $J_k(q)$ does not change if we replace the above coprimality condition with the condition $\gcd(q, b_0, b_1,\cdots, b_k) = 1$. Indeed it follows from the fact that if a polynomial $F$ satisfies $\gcd(q, b_0, b_1,\cdots, b_k) = 1$
but does not satisfies $\gcd(q, b_1, \cdots, b_k) = 1$, there is a prime $p$ dividing each of $q,b_1,b_2,\cdots, b_k$, but not $b_0$. This means, as a polynomial modulo $p$, $F(X)$ is congruent to $b_0$. Thus $F$ has no zero modulo $p$, so no root 
modulo $q$. So such an $F$ has no influence on the value of $J_k (q)$.

\begin{lemma}
\label{lm:j-power-factors-cong}
For a square-free integer $q\ge 1$  and $k=1$ we set $q=r_1 $. 
Otherwise we write  $q\ge 1$ as 
$q=r_1 \ldots r_k$ with some $k\ge 2$ positive   integers  $r_1, \ldots, r_k$ with  $\gcd(r_i,r_j) = 1$, $1 \le i < j \le k$, 
and such that 
\begin{itemize} 
\item $r_1$ is  square-free,
\item $r_i$ is $i$-th power-full but $(i+1)$-th power free when $2 \le i \le k-1$,
\item $r_k$ is   $k$-th power-full. 
\end{itemize}
Then for any fixed  integer  $k \ge 1$ we have 
\[
J_k(q) \le q^{o(1)} \prod_{i=1}^k r_i^{1-1/i} , \qquad \text{as} \ q \to \infty. 
\] 
\end{lemma}

\begin{proof} Let 
\begin{equation}\label{eq:PrimeFact}
q= p_1^{\gamma_1}\cdots p_s^{\gamma_s}
 \end{equation} 
be the prime number factorisation of $q$. 
The Chinese Remainder Theorem implies
\[
I_F (q) = I_F (p_1^{\gamma_1})\cdots I_F (p_s^{\gamma_s}),
\]
which gives immediately
\begin{equation} \label{eq:CRT}
J_k(q) \leqslant J_k\(p_1^{\gamma_1}\)\cdots J_k\(p_s^{\gamma_s}\)
 \end{equation}
 (in fact,~\eqref{eq:CRT} is an equality, but this is not needed here).

Next, for a prime $p$ and a positive integer $\gamma\le k$ we use the trivial bound 
\begin{equation} 
\label{eq:small gamma}
 J_k\(p^{\gamma}\) \le p^{\gamma-1}  J_k\(p\) \le k p^{\gamma(1-1/\gamma)} , 
\end{equation}
while for $\gamma>k$ we use a result Konyagin~\cite[Theorem~1]{Kon80}, which implies 
\begin{equation}
\label{eq:large gamma}
 J_k\(p^{\gamma}\)  \ll p^{\gamma(1-1/k)} .
\end{equation}

Substituting the bounds~\eqref{eq:small gamma} and~\eqref{eq:large gamma} 
in~\eqref{eq:CRT}, after collecting together contributions from primes which 
appear with the same exponent in~\eqref{eq:PrimeFact}, we derive
\[
J_k(q) \le C_k^s \prod_{i=1}^k r_i^{1-1/i} ,  
\] 
for some $C_k \ge 1$ which depends only on $k$. Since obviously $s! \le q$, we see that 
$C_k^s  =  q^{o(1)}$, and the result follows. 
\end{proof} 

Next, we estimate the average value of $J_k(q) $ over moduli $q \le Q$. 

\begin{lemma} 
\label{lm:Aver Jkq}
For any real $Q \ge 1$ and a fixed integer $k \ge 1$,  we have 
\[
\sum_{q \le Q} J_k(q) \le Q^{1+o(1)}, \qquad \text{as} \ Q \to \infty. 
\]
\end{lemma}

\begin{proof}
We now fix some real    numbers $R_1, \ldots, R_k\ge 1$ and consider the set 
$\cQ\(R_1,  \ldots, R_k\)$ of $q \le Q$ for which the decomposition $(r_1, \ldots,r_k)$ 
in  Lemma~\ref{lm:j-power-factors-cong} satisfies  $r_i \sim R_i$, $1 \le i \le k$, and thus 
\begin{equation} \label{eq:q in QRRR}
\begin{split} 
  \sum_{q\in \cQ\(R_1,  \ldots, R_k\)} J_k(q) 
& \le Q^{o(1)}   \sum_{q\in \cQ\(R_1,  \ldots, R_k\)}   \prod_{i=1}^k r_i^{1-1/i} \\
& \le   Q^{o(1)} \sharp  \cQ\(R_1,  \ldots, R_k\)  \prod_{i=1}^k R_i^{1-1/i} .
\end{split} 
\end{equation}

By the  classical result of  Erd{\H o}s and  Szekeres~\cite{ES34}, the number of $i$-th power-full integers $r_i \le R_i$ 
is  $O(R_i^{1/i})$. Hence, we have
\begin{equation} \label{eq:card QRRR}
\sharp  \cQ\(R_1,  \ldots, R_k\)  \ll  \prod_{i=1}^k R_i^{1/i}. 
\end{equation}
Combining~\eqref{eq:q in QRRR} and~\eqref{eq:card QRRR}, we obtain
\[ \sum_{q\in \cQ\(R_1,  \ldots, R_k\)} J_k(q) \le  Q^{o(1)}  \prod_{i=1}^k R_i .
\]

Via dyadic partition, we see that   the whole summation domain $q \le Q$ can be covered by $O((\log 2Q)^k)$
sets $\cQ(R_1,  \ldots, R_k)$  for some $R_1, \ldots, R_k$, satisfying   $R_1 \ldots R_k \ll Q$, 
the result now follows. 
\end{proof}

 \subsection{Counting solutions to congruence equations} 
For  a polynomial $f\in \Z[X]$, real  $A,B,C, U\in[1,p/2]$ and an integer $x$, denote by $\varrho(x)$ 
the number of solutions to the equation
\[
f(a)\overline{b}+c\equiv ux\bmod p
\]
in $(a,b,c,u)\in\Z^4$ with $a\sim A,\ b\sim B, \ c \sim C, \ u \sim U$.
Moreover,
for integers $x_1,x_2$, denote by
 $\vartheta(x_1,x_2)$ 
 the number of solutions to the system 
\[
\begin{cases}
f(a_1)\overline{b_1} + c\equiv ux_1\bmod p,\\ 
 f(a_2)\overline{b_2} + c\equiv ux_2\bmod p,
\end{cases}
\]
in $(a_1,a_2,b_1,b_2,c,u)\in\Z^6$ with $a_1,a_2 \sim A$,  $b_1,b_2 \sim B$,  $c\sim C$ and $u\sim U$.

First we recall the trivial inequalities
\begin{equation}\label{eq:rho-theta-1moment}
\sum_{x=1}^p \varrho(x) \ll ABCU \mand \sum_{x_1,x_2=1}^p \vartheta(x_1,x_2)\ll A^2B^2CU.
\end{equation}
In what follows, we estimate the second moment of $\varrho(x)$ and $\vartheta(x_1,x_2)$,
which is used later to prove Theorems~\ref{thm:TypeI} and~\ref{thm:TypeII}. 

\begin{lemma}\label{lm:rho-2moment}
Let  $f\in \Z[X]$ be a fixed   polynomial of degree $k\ge 2$.
Assume that $A,B,C\geqslant2$ satisfy~\eqref{eq:Cond ABC}  with a sufficiently small 
$\gamma_f > 0$
and also $U \le p/73$. Then
we have
\[
\sum_{x=1}^{p-1} \varrho(x) ^2 \le  ACU(A+B) \( \frac{(A^k+BC)BU}{p} + 1\) p^{o(1)} .
\]
\end{lemma}

\begin{proof}
Denote by $\fR$ the second moment in question. Then
\begin{align*}
\fR
&= \ssum_{\substack{a_1,a_2\sim A,\, b_1,b_2\sim B, \, c_1,c_2\sim C, \, u_1,u_2\sim U\\ (f(a_1)\overline{b_1}+c_1)u_2\equiv (f(a_2)\overline{b_2}+c_2)u_1\bmod p}}1\\
&= \ssum_{\substack{a_1,a_2\sim A,\, b_1,b_2\sim B, \, c_1,c_2\sim C, \, u_1,u_2\sim U\\ (f(a_1)+b_1c_1)b_2u_2\equiv (f(a_2)+b_2c_2)b_1u_1\bmod p}}1,
\end{align*} 
where there is an implicit restriction that $p\nmid (f(a_1)+b_1c_1)(f(a_2)+b_2c_2)$, due to the prescribed sizes of involved variables and the fact that we are summing over $1\leqslant x\leqslant p-1$.
Moreover, the above congruence equation can be rephrased by the Diophantine equation
\begin{equation}\label{eq:Diophantine-rho}
(f(a_1)+b_1c_1)b_2u_2= (f(a_2)+b_2c_2)b_1u_1 + rp,\qquad  0\leqslant |r|\leqslant R,
\end{equation} 
with
\begin{equation}\label{eq:R-bound}
0<R \ll (A^k+BC)BU/p.
\end{equation}

It is  useful to emphasise that
\begin{equation}\label{eq:maincase}
(f(a_1)+b_1c_1)(f(a_2)+b_2c_2)\neq 0
\end{equation}
according to the above restrictions. Similarly, we also have
\[
f(a_1)b_2u_2-rp\neq0,  \mand (f(a_1)+b_1c_1)b_2 u_2 -rp\neq0.
\]
We analyse the equation~\eqref{eq:Diophantine-rho} by considering divisibility. 

\underline{Step 1:} Fix $a_1$, $b_2$, $u_2$ and $r$. We see from~\eqref{eq:Diophantine-rho}  that
\[b_1\mid (f(a_1)b_2u_2-rp)\neq0.
\] 
By the divisor function bound, there are at most  $p^{o(1)}$ choices for $b_1$.

\underline{Step 2:} We now have $a_1$, $b_1$, $b_2$, $u_2$ and $r$, and further fix
$c_1$. These give at most  $p^{o(1)}$ choices for $u_1$ since
\[ 
u_1\mid ( (f(a_1)+b_1c_1)b_2 u_2 -rp)\neq0.
\]

\underline{Step 3:} 
We now define
\[
t =\frac{(f(a_1)+b_1c_1)b_2 u_2 -rp}{b_1u_1}\in\Z,
\]
which only depends on the already fixed parameters $a_1$, $b_1$, $b_2$, $c_1$, $u_1$, $u_2$ and $r$.
We are now led to count the number of solutions to the equation
\begin{equation}\label{eq:lastequation-rho} 
f(a_2)+b_2 c_2 =t,
\end{equation}
in unknowns $a_2\sim A$, $c_2\sim C$.
Interpreting~\eqref{eq:lastequation-rho} as a congruence modulo $b_2/d$, where $d \ll 1$ is the greatest common divisor of the  coefficients of $f$ and $b_2$,   the number of choices for $a_2$ is bounded by
\[
\sum_{a_2\sim A\atop f(a_2)\equiv t \bmod {b_2}} 1 \leq d \left( \frac A{b_2} +1\right) J_k (b_2/d),
\]
after which the last remaining variable $c_2$ is uniquely defined.

\underline{Step 4:} 
Recall the ranges of  $a_1$, $  c_1$, $u_2$, $r$ and also that $b_1$ and $u_1$ 
contribute  by a factor $ p^{o(1)}$ to the total count. Summing
over $b_2$,  we derive 
 \[
\fR \leq p^{o(1)} ACU(1+R) \left(\frac{A}{B}+1\right) \sum_{b_2 \sim B}\gcd(b_2,\ff) J_k(b_2/\gcd(b_2,\ff)),
\]
 where $\ff$ denotes the greatest common divisor of all coefficients of $f$.
Note that the last sum over $b_2$ is at most
 \[ \sum_{d\mid \ff}d \sum_{\substack{b_2 \sim B\\ d\mid b_2}} J_k(b_2/d)
\leq \sum_{d\mid \ff}d \sum_{b_2 \sim B/d} J_k(b_2)
\leq Bp^{o(1)}, 
\]
by Lemma~\ref{lm:Aver Jkq}. This implies
\[
 \fR\le ACU(1+R) (A+B)p^{o(1)}.
\]
Recalling the bound~\eqref{eq:R-bound} we complete the proof. 
\end{proof}

\begin{lemma}\label{lm:theta-2moment}
 Let  $f\in \Z[X]$ be a fixed polynomial of degree $k\ge 2$.
Assume that $A,B,C\geqslant2$ satisfy~\eqref{eq:Cond ABC}  with a sufficiently small $\gamma_f > 0$
and also $U \le p/73$. Then we have
\[
\sum_{x_1, x_2=1}^{p-1} \vartheta(x_1,x_2)^2
\le  A^2BCU(A+B) \( \frac{(A^k+BC)BU}{p} + 1\)^2 p^{o(1)} .
\] 
\end{lemma}

\begin{proof} 
Denote by $\fC$ the second moment in question.
Hence
 $\fC$ is bounded by the number of solutions to the system of congruences
\[
\begin{cases}
(f(a_1)\overline{b_1} + c)v\equiv (f(a_3)\overline{b_3} + d)u \bmod p,\\
(f(a_2)\overline{b_2} + c)v\equiv (f(a_4)\overline{b_4} + d)u\bmod p, 
\end{cases}
\]
with $a_1,a_2,a_3,a_4\sim A$, $b_1,b_2,b_3,b_4\sim B$, $c,d\sim C$, and $u,v\sim U$, subject to the constraint
\begin{equation}\label{eq:allcoprime to p}
p\nmid (f(a_1) +b_1 c)(f(a_2) +b_2 c)(f(a_3) +b_3 d)(f(a_4) +b_4 d).
\end{equation}
This constraint is due to the prescribed sizes of involved variables and the fact that we are summing over $1\leqslant x_1,x_2\leqslant p-1$.

Change the above congruences to  Diophantine equations:
\begin{equation} \label{eq: cong2eq}
\begin{cases} 
(f(a_1) +b_1 c)b_3v= (f(a_3) +b_3 d)b_1u+t_1p,\\
(f(a_2) +b_2 c)b_4v= (f(a_4) +b_4 d)b_2u+t_2p,
\end{cases} 
\end{equation}
where the integer variables $t_1$ and $ t_2$ satisfy $0\leqslant |t_1|,|t_2|\le  T$ with
\begin{equation}\label{eq:T-bound}
0<T\ll (A^k+BC)BU/p.
\end{equation}
Note that 
\[
(f(a_1) b_3v-t_1p)(f(a_4)b_2u+t_2p)( (f(a_1) +b_1 c)b_3v-t_1p)\neq0,
\]
given the ranges of the involved variables.

\underline{Step 1:}
We first fix $a_1$, $b_3$, $c$, $v$ and $t_1$.
It follows from the first equation of~\eqref{eq: cong2eq} that
\begin{equation} \label{eq: div}
b_1\mid(f(a_1) b_3v-t_1p)\neq0 \mand u\mid ( (f(a_1) +b_1 c)b_3v-t_1p)\neq0.
\end{equation}
This allows only at most $p^{o(1)}$ choices for $b_1$, $u$.

\underline{Step 2:}
We now have $a_1$, $b_1$, $b_3$, $c$, $u$, $v$, $t_1$.
Note that the first   equation in~\eqref{eq: cong2eq} can be reformulated as
\begin{equation} \label{eq: a3 and d}
f(a_3) +b_3 d=\frac{(f(a_1) +b_1 c)b_3v-t_1p}{b_1u},
\end{equation}
where the right hand side is now fixed already, and we seek the solutions in $a_3$, $d$.  Interpreting~\eqref{eq: a3 and d} as a congruence equation modulo $b_3$, the number of choices of $a_3$, $d$ is then at most  $O((A/B+1) \gcd(b_3,\ff) J_k(b_3/\gcd(b_3,\ff)))$,
just as argued in the proof of the preceding lemma. Here,  $\ff$ still denotes the greatest common divisor of all coefficients of~$f$.

\underline{Step 3:}
We still have $a_1$, $b_1$, $b_3$, $c$, $u$, $v$, $t_1$  and choose $a_3$, $d$ with a number of choices as discussed above. Now fix $a_4$, $b_2$, $t_2$ further.

From the second equation in~\eqref{eq: cong2eq}  we infer
\[
b_4\mid (f(a_4)b_2u+t_2p)\neq0.
\] 
This implies that $b_4$ takes at most $p^{o(1)}$ possible values. 
Finally, $a_2$ is determined up to $k$ choices given all other variables given as before.

\underline{Step 4:}
Up to now, we have considered all the 14 variables in~\eqref{eq: cong2eq}.
Collecting all above arguments, we obtain
\begin{align*}
\fC & \le \(\frac{A}{B}+1\) p^{o(1)}
\sum_{ |t_1| \le  T} \sum_{a_1 \sim A} \sum_{b_2  \sim B} \sum_{c\sim C} 
\sum_{v\sim U}  \sum_{  |t_2|\le  T} \sum_{a_4 \sim A} \\
&\qquad \qquad\qquad\qquad\qquad  \times\sum_{b_3\sim B} \gcd(b_3,\ff) J_k(b_3/\gcd(b_3,\ff)).
\end{align*}
As argued in the proof of Lemma~\ref{lm:rho-2moment}, the innermost sum over $b_3$ is at most $Bp^{o(1)}$. Hence
\[
\fC
\leqslant (1+T)^2(A+B)A^2BCUp^{o(1)}.
\]
The result  now follows immediately by recalling the bound~\eqref{eq:T-bound} for $T$.
\end{proof}

\section{Proof of Theorem~\ref{thm:TypeI} and Corollary~\ref{cor:TypeI}}
\label{sec:TypeI-proof}

\subsection{Amplification via bilinear shifts}  
\label{sec:bilin shift}
Recalling the definition~\eqref{eq:charactersum} we write 
 \[
\cC_{f,\chi} (\balpha,\bbeta) =\sum_{a \sim A} \sum_{b \sim B} \alpha(a,b)\chi(b) \sum_{c \sim C} \chi(f(a) \overline{b} +  c).
\]
 We fix two  parameters $U, V\ge 1$ with 
\[
 UV = \frac{1}{4} C. 
\]

We may attach a weight function to mimic the indicator function of the interval $[C,2C]$, and as argued in~\cite[Equations~(4.3) and~(4.4)]{FM98}, we make a shift $c\rightarrow c+uv$ with $u\sim U$ and $v\sim V$.
By Fourier inversion, we then arrive at
\begin{equation}
\label{eq:S T}
\cC_{f,\chi} (\balpha,\bbeta) \ll \frac{\log p}{UV} \fS,
 \end{equation}  
where, for some $\zeta \in \R$, 
\begin{align*}
\fS & = \sum_{a \sim A} \sum_{b \sim B}  \sum_{c \sim C} \sum_{u \sim U}
\Bigl| \sum_{v \sim V} \chi(f(a)\overline{b}+c+uv) \e(\zeta v)\Bigr|\\
& = \sum_{a \sim A} \sum_{b \sim B}  \sum_{c \sim C} \sum_{u \sim U}
\Bigl| \sum_{v \sim V} \chi((f(a)\overline{b}+c)\overline{u}+v) \e(\zeta v)\Bigr|.
\end{align*}
Note that the terms with $f(a)+bc=0$ contributes to $\fS$ at most $O(AUVp^{o(1)})$. We henceforth consider the contributions from $f(a)+bc\neq0$, in which case we automatically have $p\nmid (f(a)+bc)$.

For $x\in\Z$, denote by $\varrho(x)$ the number of solutions to the equation
\[
f(a)\overline{b}+c\equiv ux\bmod p
\]
in $(a,b,c,u)\in\Z^4$ with $a\sim A,\ b\sim B, \ c \sim C, \ u \sim U$.
Hence
\[
\fS \leqslant  \sum_{x=1}^{p-1} \varrho(x) \Bigl| \sum_{v \sim V} \chi(x+v) \e(\zeta v)\Bigr|+AUVp^{o(1)}.
\]
From H\"older's inequality, it follows that
\[
\fS  \leqslant \fS_1^{1-\frac{1}{r}}(\fS_2 \fS_3)^{\frac{1}{2r}}+AUVp^{o(1)}
\]
for all $r\in\Z^+$,
where
\[
\fS_j  = \sum_{x=1}^{p-1} \varrho(x)^j,\qquad  j=1,2,
\]
and
\[
\fS_3  = \sum_{x=1}^{p-1} \Bigl| \sum_{v \sim V} \chi(x+v) \e(\zeta v)\Bigr|^{2r}.
\]

We are now in a good position to apply~\eqref{eq:rho-theta-1moment} to $\fS_1$, Lemma~\ref{lm:rho-2moment} to $\fS_2$ and Lemma~\ref{lm:DavenportErdos} to $\fS_3$, so that
\begin{align*}
\fS \leqslant  (ABCU)^{1-\frac{1}{2r}}(V^rp+V^{2r}p^{\frac{1}{2}})^{\frac{1}{2r}} \(\frac{A}{B}+1\)^{\frac{1}{2r}} \( \frac{(A^k+BC)BU}{p} + 1\)^{\frac{1}{2r}}&p^{o(1)}\\
+AU&Vp^{o(1)}.
\end{align*}

\subsection{Optimisation and deriving Theorem~\ref{thm:TypeI}} 
Taking 
\[
V=p^{\frac{1}{2r}} \mand U=\frac{1}{4}Cp^{-\frac{1}{2r}}
\] 
(and noting that  $U, V  \ge 1$), we arrive at
\begin{align*}
|\cC_{f,\chi} (\balpha,\bbeta) |
& \le  (ABC)^{1-\frac{1}{2r}}U^{-\frac{1}{2r}}p^{\frac{1}{4r}+o(1)}\(\frac{A}{B}+1\)^{\frac{1}{2r}} \( \frac{(A^k+BC)BU}{p} + 1\)^{\frac{1}{2r}} +Ap^{o(1)}\\
& \le  (AB)^{1-\frac{1}{2r}}C^{1-\frac{1}{r}}p^{\frac{r+1}{4r^2}+o(1)}\(\frac{A}{B}+1\)^{\frac{1}{2r}} \( \frac{(A^k+BC)BC}{p^{1+\frac{1}{2r}}} + 1\)^{\frac{1}{2r}} +Ap^{o(1)}\\
& \le  (AB)^{1-\frac{1}{2r}}C^{1-\frac{1}{r}}p^{\frac{r+1}{4r^2}+o(1)}\(\frac{A}{B}+1\)^{\frac{1}{2r}} \( \frac{(A^k+BC)BC}{p^{1+\frac{1}{2r}}} + 1\)^{\frac{1}{2r}}.
\end{align*}
This completes the proof of Theorem~\ref{thm:TypeI}.

\subsection{Deriving  Corollary~\ref{cor:TypeI}}
Taking $A, B, C = p^{\eta+o(1)}$ in Theorem~\ref{thm:TypeI}, yields
\[
\Delta_1 \le p^{-\zeta + o(1)}
\]
where 
\begin{align*}
\zeta & = \frac{\eta}{r} + \frac{\eta}{r} -  \frac{r+1}{4r^2} -  \frac{1}{2r} \max\{(k+2)\eta  -1 -\frac{1}{2r}, 0\}  \\
& = \frac{2\eta}{r} -  \frac{r+1}{4r^2} -  \frac{1}{2r} \max\{(k+2)\eta  -1 -\frac{1}{2r}, 0\}  .
\end{align*}
Thus, either 
\[
\zeta  = \frac{2\eta}{r} -  \frac{r+1}{4r^2} -  \frac{1}{2r}\((k+2)\eta  -1 -\frac{1}{2r}\) 
=   \frac{1}{4r}  -  \frac{(k-2)\eta}{2r}  ,
\]
or 
\[
\zeta  = \frac{2\eta}{r} -  \frac{r+1}{4r^2} 
=   \frac{8\eta-1}{4r} -  \frac{1}{4r^2} .
\]
In either case $\zeta > 0$ as long as the condition~\eqref{eq:Cond N1}
holds and $r$ is large enough.

  \section{Proof of Theorem~\ref{thm:TypeII} and Corollary~\ref{cor:TypeII}}
  \label{sec:TypeII-proof}
  
 \subsection{Amplification via bilinear shifts}   
Assume that $\|\balpha\|_\infty,\|\bbeta\|_\infty \leqslant 1$.
Before we apply bilinear shifts, as in Section~\ref{sec:bilin shift} we use 
 the Cauchy--Schwarz inequality to remove absolute values. Thus,  we have
\begin{align*}
|\cC_{f,\chi} (\balpha,\bbeta)|^2
&\leqslant C\sum_{c\sim C}\Bigl|\sum_{a \sim A} \sum_{b \sim B}\alpha(a,b) \chi(f(a) + bc)\Bigr|^2\\
&= C\sum_{a_1,a_2 \sim A} \sum_{b_1,b_2 \sim B}\alpha(a_1,b_1)\overline{\alpha(a_2,b_2)}   \sum_{c\sim C}\chi(f(a_1) + b_1c)\overline{ \chi(f(a_2) + b_2c)}.
\end{align*}

Following a similar argument as in Section~\ref{sec:TypeI-proof}, 
we introduce the shift $c\rightarrow c+uv$, and sum over $u\sim U$, $v\sim V$ with
\[
UV=\frac{1}{4}C,
\]
so that
\begin{align*}
|\cC_{f,\chi} (\balpha,\bbeta)|^2
&\ll \frac{C\log p}{UV}\sum_{a_1,a_2 \sim A} \sum_{b_1,b_2 \sim B}\sum_{c\sim C}\sum_{u\sim U}\\
&\qquad \qquad \qquad \quad  \times \Bigl|\sum_{v\sim V}\e(\zeta v)\chi(f(a_1)\overline{b_1} + c+uv)\overline{ \chi(f(a_2) \overline{b_2} + c+uv)}\Bigr|\\
&= \frac{C\log p}{UV}\sum_{a_1,a_2 \sim A} \sum_{b_1,b_2 \sim B}\sum_{c\sim C}\sum_{u\sim U}\\
&\qquad \qquad \qquad \quad  \times \Bigl|\sum_{v\sim V}\e(\zeta v)\chi((f(a_1)\overline{b_1} + c)\overline{u}+v)\overline{ \chi((f(a_2)\overline{b_2} + c)\overline{u}+v)}\Bigr|, 
\end{align*}
for some $\zeta\in\R$. 
If $f(a_1)+b_1c=0$, then there are at most $p^{o(1)}$ choices for $b_1$, $c$ after $a_1\sim A$ is given. Hence the terms with $(f(a_1)+b_1c)(f(a_2)+b_2c)=0$ 
contribute to the  right hand side at most $A^2BCp^{o(1)}$.
We henceforth consider the terms with $(f(a_1)+b_1c)(f(a_2)+b_2c)\neq0$, which yields 
$p\nmid (f(a_1)+b_1c)(f(a_2)+b_2c)$ automatically,  given the sizes of all variables.

For $x_1,x_2\in\Z$, denote by $\vartheta(x_1,x_2)$ the number of solutions of the system
\[
f(a_1)\overline{b_1} + c\equiv ux_1\bmod p,\qquad f(a_2)\overline{b_2} + c\equiv ux_2\bmod p,
\]
in  the variables $a_1,a_2 \sim A$, $b_1,b_2 \sim B$,  $c\sim C$ and $u\sim U$. In this way, we may write
\[
|\cC_{f,\chi} (\balpha,\bbeta)|^2
\ll \frac{C\log p}{UV}\sum_{x_1,x_2=1}^{p-1}\vartheta(x_1,x_2)\Bigl|\sum_{v\sim V}\chi(x_1+v)\overline{ \chi(x_2+v)}\Bigr|  +A^2BCp^{o(1)}.
\]

From H\H{o}lder's inequality, it follows that
\[
|\cC_{f,\chi} (\balpha,\bbeta)|^2
\ll \frac{C\log p}{UV}\fY_1^{1-\frac{1}{r}}(\fY_2\fY_3)^{\frac{1}{2r}}+ A^2BCp^{o(1)}
\]
with
\[
\fY_j=\sum_{x_1,x_2=1}^{p-1}\vartheta(x_1,x_2)^j,\qquad  j=1,2,
\]
and
\[
\fY_3=\sum_{x_1,x_2=1}^{p-1}\Bigl|\sum_{v\sim V}\chi(x_1+v)\overline{ \chi(x_2+v)}\Bigr|^{2r}.
\]

We are now in a good position to apply~\eqref{eq:rho-theta-1moment} to $\fY_1$, Lemma~\ref{lm:theta-2moment} to $\fY_2$ and Lemma~\ref{lm:DavenportErdos2} to $\fY_3$, so that
\begin{align*}
|\cC_{f,\chi} (\balpha,\bbeta)|^2
&\leqslant \frac{Cp^{o(1)}}{UV}(A^2B^2CU)^{1-\frac{1}{2r}}(V^rp^2+V^{2r}p)^{\frac{1}{2r}} \(\frac{A}{B}+1\)^{\frac{1}{2r}}\\
&\qquad \qquad \qquad \qquad \qquad \qquad \quad  \times \( \frac{(A^k+BC)BU}{p} + 1\)^{\frac{1}{r}}+ A^2BCp^{o(1)}.
\end{align*}

\subsection{Optimisation and deriving Theorem~\ref{thm:TypeII}} 
Taking 
\[V=p^{\frac{1}{r}} \mand U=\frac{1}{4}Cp^{-\frac{1}{r}},
\]
we get
\begin{align*}
|\cC_{f,\chi} (\balpha,\bbeta)|^2
&\leqslant (ABC)^{2-\frac{1}{r}}p^{\frac{1}{2r}+\frac{1}{2r^2}+o(1)}\(\frac{A}{B}+1\)^{\frac{1}{2r}}\\
&\qquad \qquad \qquad \qquad \times \( \frac{(A^k+BC)BU}{p} + 1\)^{\frac{1}{r}}+A^2BCp^{o(1)}\\
&\leqslant (ABC)^{2-\frac{1}{r}}p^{\frac{1}{2r}+\frac{1}{2r^2}+o(1)}\(\frac{A}{B}+1\)^{\frac{1}{2r}}\\
&\qquad \qquad \qquad \qquad \times \( \frac{(A^k+BC)BC}{p^{1+\frac{1}{r}}} + 1\)^{\frac{1}{r}}+A^2BCp^{o(1)}\\
&\leqslant (ABC)^{2-\frac{1}{r}}p^{\frac{1}{2r}+\frac{1}{2r^2}+o(1)}\(\frac{A}{B}+1\)^{\frac{1}{2r}}\\
&\qquad \qquad \qquad \qquad  \times \( \frac{(A^k+BC)BC}{p^{1+\frac{1}{r}}} + 1\)^{\frac{1}{r}}.
\end{align*}
This completes the proof of Theorem~\ref{thm:TypeII}.

\subsection{Deriving  Corollary~\ref{cor:TypeII}}
In this case, taking $A, B, C = p^{\eta+o(1)}$ in Theorem~\ref{thm:TypeII}, yields
\[
\Delta_2 \le p^{-\zeta + o(1)}
\]
where 
\[
\zeta  = \frac{3\eta}{2r}  -  \frac{r+1}{4r^2} -  \frac{1}{2r} \max\{(k+2)\eta  -1 -\frac{1}{r}, 0\}  .
\]
Thus, either 
\[
\zeta  = \frac{3\eta}{2r} -  \frac{r+1}{4r^2} -  \frac{1}{2r}\((k+2)\eta  -1 -\frac{1}{r}\) 
=   \frac{r+1}{4r^2}  -  \frac{(k-1)\eta}{2r}  ,
\]
or 
\[
\zeta  = \frac{3\eta}{2r} -  \frac{r+1}{4r^2} 
=   \frac{6\eta-1}{4r} -  \frac{1}{4r^2} .
\]
In either case $\zeta > 0$ as long as the condition~\eqref{eq:Cond N2}
holds and $r$ is large enough.

\section{Comments} 

As a by-product, Ganguly and Rajan~\cite{GR23} 
initiated the study of $2\times2$ matrices which have irreducible characteristic polynomials. One may ask similar questions for $n\times n$ matrices with $n\geqslant3$, but we do not know how to approach this via character sums as argued in the current paper.

Denote by $\cI_n(p)$ the set of all $n\times n$ matrices with entries in $\F_p$, of 
irreducible characteristic polynomial, and by $\cI_n(N,p)$ the set consisting all elements in $\cI_n(p)$ with entries falling into $[1,N]$.
While there is no feasible approach to estimating $\cI_n(N,p)$ via character sums, one can notice that the irreducibility 
of the characteristic polynomial of any square matrix $A$ over any field is preserved under scalar multiplication $A \mapsto \lambda A$ 
with any $\lambda \ne 0$.
This allows us to use classical ideas of Korobov~\cite{Kor63,Kor67} on the distribution of vectors $(\lambda a_1, \ldots, \lambda a_k)$, $\lambda \in \F_p^*$, 
on average over vectors  $( a_1, \ldots,   a_k) \in \F_p^k$ (applied with $k = n^2$ in our setting).  It is reasonable to  expects that this 
approach yields the bound of the 
form 
\[
\sharp \cI_n(N,p)  = \frac{(2N)^{n^2}}{p^{n^2}} \sharp \cI_n(p) + O(p^{n^2-1}( \log p)^{n^2}).
\]
We also note that it is easy to derive from a very special case of 
a result of Reiner~\cite[Theorem~2]{Re61} that $\sharp \cI_n(p) = n^{-1} p^{n^2} + O(p^{n^2-1})$.

\section{Acknowledgements}
This work was completed during the Program ``Fouvry-73'' held in the Bernoulli Centre of EPF Lausanne during the summer of 2026. We thank the organisers of this four-week program, and the hospitality and wonderful working environment are greatly appreciated. I.S. would also like to thank the Max Planck Institute for 
Mathematics, Bonn, for its support and excellent working conditions during the
preparation of this work.

During the preparation of this work, I.S. was supported in part by the  Australian Research Council Grant  DP230100534, and P.X. by Shaanxi NSF (No. 2025JC-QYCX-002) and Shaanxi Fundamental Science Research Project for Mathematics and Physics (No.25JSZ007).

\end{document}